\documentclass[11pt, final]{article}
\usepackage[utf8]{inputenc}
\usepackage[english]{babel}
\usepackage{csquotes}
\usepackage{comment}
\usepackage{amsmath}
\usepackage{amsthm}
\usepackage{amssymb}
\usepackage{commath}
\usepackage{derivative}
\usepackage[mathcal]{eucal}
\let\mathscr\mathcal
\usepackage{color}
\usepackage{tensor}
\usepackage{breqn}
\usepackage{hyperref}
\usepackage[a4paper, total={6in, 9in}]{geometry}

\newcommand{\supnorm}[1]{\left\lVert #1 \right\rVert_{\infty}}

\newcommand{\pnorm}[2]{\left\lVert #2 \right\rVert_{#1}}

\newtheorem{theorem}{Theorem}[section]

\newtheorem{lemma}[theorem]{Lemma}
\newtheorem{definition}{Definition}

\title{Front Tracking for Scalar Conservation Laws with Spatially Heterogeneous Flux}
\author{Parasuram Venkatesh\thanks{Centre for Applicable Mathematics, Tata Institute of Fundamental Research, e-mail: venkatesh2020@tifrbng.res.in}}
\date{}

\begin{document}
\maketitle
\begin{abstract}
	In this article, we propose a novel front tracking scheme for scalar conservation laws with spatially heterogeneous, uniformly convex flux and prove that approximations converge to the unique entropy solution. The main tools are Dafermos' generalised characteristics and Kruzkov's entropies. Crucially, our method handles fluxes where classical theory fails completely. As a concrete demonstration, we construct entropy solutions for a Cauchy problem with flux $f(x,u)=xu^2$, where bounded initial data can become unbounded in finite time, even on compact spatial domains. This finite-time blow-up violates the maximum principle, rendering all classical existence techniques—based on $L^{\infty}$ estimates and compactness—inapplicable. However, the flux $f(x,u(x,t))$ remains bounded despite $u$ blowing up, and our front tracking scheme exploits this to construct approximations that converge to an entropy solution.
	
	\smallskip
	
	
	\smallskip
	
\end{abstract}
\tableofcontents

\section{Introduction}
Consider the scalar conservation law
\begin{equation}\label{claw}
	\begin{split}
		u_t+f(x,u)_x&=0, \\
		u(x,0)&=u_0(x),
	\end{split}
\end{equation}
in $\mathbb{R}\times[0,\infty)$. An $L^1$ semigroup approach provided the initial breakthrough for such problems in the work of Kruzkov \cite{kruzkov}, who proved existence, uniqueness, and stability of `entropy' solutions under some regularity assumptions on the flux. In this framework, however, although bounded solutions are unique when they exist, Kruzkov's vanishing viscosity construction cannot guarantee the existence of solutions if the mixed partial derivative $f_{xu}$ is not uniformly bounded below. Even for the simplest convex heterogeneous fluxes such as $f(x,u)=(2+\sin x)u^2$, we have that $f_{xu}(x,u)=2u\cos x$, which has no uniform lower bound. The classical theory of scalar conservation laws can be found in \cite{Dafermos2016}.

Fluxes with spatial dependence have been analysed in the literature, such as on Riemannian or Lorentzian manifolds by LeFloch and Ben-Artzi in \cite{hypconlawmanifolds}, motivated by geophysical and relativistic applications. Lengeler and Müller \cite{Lengeler} proved the $L^1$ contraction property for scalar conservation laws on closed manifolds. The Cauchy problem was studied as a non-linear transport equation by Felix Otto in \cite{otto}, who also derived regularising effect of the non-linearity. Colombo, Perrollaz, and Sylla \cite{ConLawHJB} studied well-posedness for such equations when $f$ is ``non-linear'' in the conserved variable $u$, i.e. the set of zeroes of $f_{uu}$ is of measure zero, and the flux satisfies a super-linear growth assumption. This includes fluxes like $f(x,u)=(2+\sin x)u^2$. Convergence of a numerical scheme in this setting was explored by Abraham Sylla using a finite volume scheme \cite{sylla2024convergencefinitevolumescheme}, discretising the heterogeneity and leveraging the theory of conservation laws with discontinuous flux \cite{discofluxog}. Asymptotic emergence of simple shocks and their $L^2$ stability was analysed by Ghoshal and Venkatesh in \cite{singleshock}. Recently, a mild generalisation of Kruzkov's theory was obtained for scalar conservation laws in arbitrary dimension with spatial heterogeneity by Paz Hashah \cite{hashash}. A purely $L^1$ approach, based on the Lions-Perthame-Tadmore kinetic formulation \cite{perthame} was developed by Anne-Laure Dalibard for equations with heterogeneous flux in \cite{dalibard}.

Laws of this form are also appropriate models, for example, for traffic flow with changing road conditions, such as varying maximum velocity. In this article, we focus on the case where $f$ is uniformly convex in $u$, and develop a front-tracking approach to the Cauchy problem \eqref{claw}. The same results hold for uniformly \textit{concave} fluxes as well by a change of variables from $x$ to $-x$.
 
We work with the heterogeneity directly in order to construct a sequence of approximate solutions, taking inspiration from Dafermos' theory of generalised characteristics for scalar conservation laws with strictly convex flux \cite{GenChar} and the `front tracking' method for scalar conservation laws more generally \cite{polygon}. It was also implemented as a numerical method in \cite{polgyonnumeric}; see also the article by Holden and Holden \cite{HoldenHolden}.

This technique has been extended to systems as well; it was first used by DiPerna in \cite{DIPERNA1976187}, and implemented as a numerical scheme for gas dynamics by Swartz and Wendroff \cite{aztek}. Front tracking was also used to solve the general Cauchy problem for small total variation data in \cite{BRESSAN1992414} and \cite{Risebro1993AFA}, as an alternative to Glimm's random choice scheme \cite{Glimm1965SolutionsIT}. Before this, classical techniques only yielded local-in-time existence of solutions \cite{Kato1975}.

Standard reference texts for front tracking as applied to the Cauchy problem for systems of conservation laws in one spatial dimension include \cite{BressBook} and \cite{Holden2015}. The results all generally assume a spatially homogeneous flux. Heterogeneity typically takes the form of sharp discontinuities in otherwise spatially homogeneous fluxes, and the front tracking method has been adapted to solve these Cauchy problem as well \cite{ftdisco}. Such equations arise naturally in e.g. traffic flow, two-phase flows for oil extraction, etc. and a brief overview of the theory can be found in \cite[Chapter 8]{Holden2015}. Conservation laws with spatially discontinuous flux functions have been studied in several important papers; a necessarily incomplete list includes  \cite{discofluxog}, \cite{Andreianov2011}, \cite{diehl}, \cite{BKRT2004}, \cite{ghoshal2015optimal}, \cite{regidscoflux}, \cite{optentropy}, the references contained therein, and others.

However, to the best of this author's knowledge, there has been no work yet adapting the front tracking method to cases involving fluxes with smoothly varying heterogeneity, and more generally independent of the discontinuous flux theory. This article is intended to address the gap.

The structure of this article is as follows. Section~\ref{prelim} goes over some of the preliminary material for scalar conservation laws, including the entropy inequality for \eqref{claw} and the theory of generalised characteristics. Section~\ref{ass} outlines our structural assumptions on the flux. Section~\ref{pw const} details the front tracking algorithm for piecewise `stationary' initial data and establishes a priori bounds, while Section~\ref{cauchy prob} concludes well-posedness of the general Cauchy problem \eqref{claw}, i.e. existence of solutions satisfying entropy conditions \eqref{Het entropy}, or equivalently \eqref{entropy}. In Section~\ref{app}, we demonstrate the value of our method by proving the existence of entropy solutions with fluxes that are not contained in any of the existing theory.

All functions of space and time are assumed to be \textit{càdlàg} concerning the spatial variable, i.e. right-continuous with left limits, whenever the spatial traces exist unless otherwise specified.

\subsection{Preliminaries}\label{prelim}
The classical text \cite{Dafermos2016} is a comprehensive reference for the literature on conservation laws. Here, we recall some of the basic concepts that we employ for our work. Differentiable solutions of \eqref{claw} can be defined locally in time \cite{Kato1975} for sufficiently regular initial data, but may break down in finite time. This motivates a notion of `weak solution' that can be defined globally in time. The space-time divergence form of the equation naturally yields the following definition: $u\in L^1_{loc}([0,\infty);\mathbb{R})$ is a weak solution of the initial value problem \eqref{claw} if, for all $\varphi\in C_c^{\infty}\left(\mathbb{R}\times[0,\infty)\right)$:
\begin{equation}\label{weak}
	\iint_{\mathbb{R}^2_+}u(x,t)\varphi_t(x,t)+f(x,u(x,t))\varphi_x(x,t)dx=\int_{-\infty}^{\infty}u_0(x)\varphi(x,0)dx.
\end{equation}
However, there are infinitely many functions $u$ satisfying \eqref{weak}. Hence, a selection criterion is introduced. `Entropy' weak solutions of \eqref{claw} in this framework are those satisfying a family of inequalities of the form
\begin{equation}\label{Het entropy}
	\eta(u)_t+Q(x,u)_x+\eta^{\prime}(u)f_x(x,u)-Q_x(x,u)\leq0,
\end{equation}
in the sense of distributions over $\mathscr{D}(\mathbb{R}\times[0,\infty))$ for all pairs of functions $\eta,Q$ such that the `entropy' $\eta$ is a convex function and the `entropy flux' $Q(x,\cdot)$ is an antiderivative of $\eta^{\prime}(\cdot)f_u(x,\cdot)$ for each fixed $x$. Note that given $\eta$, we can always construct $Q$ that satisfies this requirement. More concretely, for each such entropy-entropy flux pair $\eta,Q$ and smooth, non-negative $\varphi(x,t)$ compactly supported in $\mathbb{R}\times[0,\infty)$, an entropy-admissible weak solution must be such that
\[
\iint\eta(u)\varphi_t+Q(x,u)\varphi_x dxdt\geq\iint\varphi \left(\eta^{\prime}(u)f_x(x,u)-Q_x(x,u)\right)dxdt-\int\eta(u_0)\varphi(x,0)dx.
\]
The entropy solution can be equivalently characterised in the following form, which we use for our purposes: $u$ is an entropy solution of \eqref{claw} with $u_0\in L^{\infty}(\mathbb{R})$ if, for any $k\in\mathbb{R}$, we have that for all $\varphi\in C_c^{\infty}(\mathbb{R}\times[0,\infty))$ such that $\varphi\geq0$,
\begin{equation}\label{entropy}
	\begin{split}
		&\iint\abs{u(x,t)-k}\varphi_t+\operatorname{sgn}(u(x,t)-k)\left[f(x,u(x,t))-f(x,k)\right]\varphi_xdxdt \\
		\geq&\iint\operatorname{sgn}(u(x,t)-k)f_x(x,k)\varphi(x,t) dxdt-\int\abs{u_0-k}\varphi(x,0)dx.
	\end{split}
\end{equation}
This can be interpreted as entropy equations for the particular family of entropies $\eta(u)=\abs{u-k}$, which is convex though not strictly so. However, many fluxes of interest may not satisfy Kruzkov's original assumptions. We refer to \cite{ConLawHJB},\cite{singleshock} for examples of such spatially heterogeneous fluxes. 

Let us briefly recall Dafermos' theory of `generalised characteristics' for scalar conservation laws with convex flux \cite{GenChar}. A natural starting point is the classical method of characteristics. Differentiating \eqref{claw}, we obtain the quasilinear equation
\begin{equation}\label{quasi claw}
	u_t+f_u(x,u)u_x=-f_x(x,u).
\end{equation}
The method of characteristics applied to \eqref{quasi claw} yields the system of ODEs
\begin{align}
	\dot{y}(s)&=f_u(y(s),z(s)),\label{char ode} \\
	\dot{z}(s)&=-f_x(y(s),z(s)),
\end{align}
where $q(s)$ is the characteristic trajectory, and $p(s)$ is the value function along the characteristic. So far, no convexity assumptions are required, and it can be shown under quite general assumptions that classical solutions of the Cauchy problem \eqref{claw} exist at least locally in time for smooth initial data, as in Kato's seminal paper \cite{Kato1975}.

Working backward from a \textit{given} an entropy solution of \eqref{claw} with convex flux $f$, we have that from every point $(x,t)$ with $t>0$, we can define a unique \textit{forward characteristic} $y_f:[t,\infty)\to\mathbb{R}$ and a non-empty set of \textit{backward characteristics} $y_b:[0,t]\to\mathbb{R}$, i.e. Lipschitz curves with $y(t)=x$ solving the differential inclusion
\begin{equation}\label{DafDiff}
	\dot{y}(s)\in[f_u(y(s),u(y(s)+,s)),f_u(y(s),u(y(s)-,s))]
\end{equation}
on their respective domains, where $u(x\pm,t)$ respectively denote the left and right traces in space of $u$ at $(x,t)$. Since $f$ is convex in the second variable, $f_u$ is monotonically increasing, and entropy solutions of $\eqref{claw}$ satisfy the inequality $u_-\geq u_+$, hence the interval in \eqref{DafDiff} is well-defined for all $(x,t)$. Forward characteristics are also defined for points of the form $(x,0)$, but they may not be unique. A simple computation tells us that $f$ must be conserved along the trajectories of \eqref{char ode}
\begin{equation}\label{f conserved}
	\begin{split}
		\frac{d}{dt}f(y(t),z(t))&=f_x(y(t),z(t))\dot{y}(t)+f_u(y(t),z(t))\dot{z}(t) \\
		&=f_x(y(t),z(t))f_u(y(t),z(t))-f_u(y(t),z(t))f_x(y(t),z(t)) \\
		&=0.
	\end{split}
\end{equation}
At points of continuity of $u$, of course, $\dot{y}$ only has one permissible value, but even on points of discontinuity, it can be shown that $\dot{y}$ has a determinate value. That is, $y:[t_0,T]\to\mathbb{R}$ be a Lipschitz solution to \eqref{DafDiff} for some $t_0\geq0$. Then, for almost all $t\in[t_0,T]$, we have that
\[
\dot{y}(t)=
\begin{cases}
	&f_u(y(t),u(y(t),t))\text{ if }u(y(t)-,t)=u(y(t)+,t), \\[1ex]
	\smallskip
	&\dfrac{f(y(t),u(y(t)-,t))-f(y(t),u(y(t)+,t))}{u(y(t)-,t)-u(y(t)+,t)}\text{ if }u(y(t)-,t)>u(y(t)+,t).
\end{cases}
\]
The extremal backward characteristics from any point $(x,t)$ with $t>0$ are `genuine' characteristics, i.e. $u$ is continuous at $(y_{\pm}(s),s)$ for $s\in(0,t)$. Dafermos' theory of generalised characteristics, as laid out in \cite{GenChar}, takes for granted the existence of an entropy solution of \eqref{claw} with left and right spatial traces at all positive times. In this article, we work backwards and derive the existence of entropy solutions using the generalised characteristics themselves.

\subsection{Structural assumptions}\label{ass}
We make the following assumptions about the flux $f$:
\begin{flalign}
	\label{S0}\tag{$S_0$}&\mbox{\textbf{Stationarity at Zero:}}\text{ for  all }q\in\mathbb{R}:f(q,0)\equiv f_u(q,0)\equiv0.&& \\
	\label{C2}\tag{$C^2$}&\mbox{\textbf{Smoothness:}}\text{ }f\in C^2(\mathbb{R}^2;\mathbb{R}). \\
	\label{UC}\tag{UC}&\mbox{\textbf{Uniform Convexity:}}\text{ }f_{uu}\geq\alpha>0. \\
	\label{FSP}\tag{FSP}&\mbox{\textbf{Finite Speed of Propagation:}}\text{ }\theta(v)=\sup_{x\in\mathbb{R}}\abs{f_u(x,v)}\in C(\mathbb{R}).
\end{flalign}
The assumptions are reminiscent of \cite{singleshock}; here, the stationarity condition at two points is replaced with one at a single point, but also involving a spatial derivative. The assumption \eqref{FSP} is used in place of the `compact non-homogeneity' condition of \cite{ConLawHJB}, and serves a similar purpose. Nagumo growth is not required as in \cite{singleshock}, but is implied by the assumptions. In particular, \eqref{S0} and \eqref{UC} together imply that $f(x,u)\geq\alpha u^2/2$ independent of $x$, and for all $u$.

\section{Piecewise stationary data}\label{pw const}
Under the structural assumptions \eqref{S0}-\eqref{C2}-\eqref{UC}-\eqref{FSP} on the flux, we demonstrate the existence of solutions to the Cauchy problem \eqref{claw} for $BV$ initial data by a front tracking argument. First, a simple lemma.
\begin{lemma}[Existence of stationary solutions]\label{stat sol}
	Suppose $f$ satisfies \eqref{S0}-\eqref{C2}-\eqref{UC}-\eqref{FSP}. Then for all $a> 0$, there exist two global-in-time, classical, stationary solutions $u_a^{\pm}$ of \eqref{claw}, i.e. $\partial_tu_a^{\pm}=0$, such that $f(x,u_a^{\pm}(x,t))\equiv a$. If $a=0$, then $u_a\equiv0$.
\end{lemma}
\begin{proof}
	From \eqref{S0}, it is clear that the trivial zero function is a classical solution of the initial value problem \eqref{claw}. Since $f(x,u)>0$ for all $u\neq0$ and $x\in\mathbb{R}$, this is the only stationary solution corresponding to $a=0$. For $a>0$, note that $a=f(0,b_{\pm})$ for some $b_{-}<0<b_{+}$. Since $f_u(0,b_{\pm})\neq0$, the existence of $u_{a}^{\pm}$ follows from a simple application of the implicit function theorem. By \eqref{FSP}, $\abs{u_a^{\pm}}$ are strictly bounded below, away from zero. 
\end{proof}
With these globally defined stationary solutions in hand, we turn to a generalised form of the Riemann problem, which serves as the building block for our front tracking algorithm. This is analogous to the role that the classical Riemann problem plays in front tracking for \eqref{claw} with spatially homogeneous flux.

Define the mapping $g:\mathbb{R}^2\to\mathbb{R}$ by
\begin{align}\label{g}
	g(x,u)=\operatorname{sgn}(u)f(x,u).
\end{align}
Note that $g$ is continuous and strictly monotone, thus continuously invertible concerning $u$, and that its derivative vanishes only at $u=0$. Stationary solutions in the sense of Lemma~\ref{stat sol} now correspond to those functions $v$ such that $g(x,v(x))$ is a constant function of $x$. The positive and negative stationary solutions corresponding to a level $a\geq0$ correspond to $v_{\pm}$ such that $g(x,v_{\pm}(x))=\pm a$ respectively, and coincide with the trivial zero function if $a=0$.

\subsection{Generalised Riemann problem}
Without space dependence, the stationary solutions are precisely the constant ones. However, under our assumptions, $0$ is in general the only constant that is also a stationary solution. Thus, Riemann-type initial data can no longer be considered as simple ``building blocks'' that can be leveraged to generate approximate solutions. Instead, we turn to the generalised characteristics, and exploit the flux conservation along trajectories of genuine characteristics derived in \eqref{f conserved}.

We say that $\overline{u}(x)$ is of `generalised Riemann form' if, for some $\overline{x}\in\mathbb{R}$ and (distinct) stationary solutions $u_l,u_r$ in the sense of Lemma~\ref{stat sol}, we have that
\begin{align}\label{gen riem data}
	\overline{u}(x)=
	\begin{cases}
		u_l(x)&\text{ if }x<\overline{x}, \\
		u_r(x)&\text{ if }x>\overline{x}.
	\end{cases}
\end{align}
In terms of the mapping $g$ as defined above in \eqref{g}, we can say that generalised Riemann data are precisely those of the form
\begin{align}\label{g gen riem data}
	g(x,\overline{u}(x))=
	\begin{cases}
		g_l&\text{ if }x<\overline{x}, \\
		g_r&\text{ if }x>\overline{x},
	\end{cases}
\end{align}
for some $g_l\neq g_r$. We exclude the trivial case $g_l=g_r$ since it exactly corresponds to a stationary solution as already described. Given any $\overline{g}\in\mathbb{R}$, define $U[\overline{g}]$ to be the unique $C^1$ function such that $g(x,U[\overline{g}](x))\equiv\overline{g}$. This correspondence is continuous in the following sense.
\begin{lemma}[Uniform inversion bounds]\label{cont dependence}
	Given $g_1,g_2\in\mathbb{R}$ both either positive or negative, the corresponding stationary solutions $U[g_1],U[g_2]$ are such that
	\[
	\supnorm{U[g_1]-U[g_2]}\leq\sqrt{\dfrac{2}{\alpha}\abs{g_1-g_2}}.
	\]
	If $g_1>0>g_2$, then
	\[
	\supnorm{U[g_1]-U[g_2]}\leq\sqrt{\dfrac{2}{\alpha}}\left\{\sqrt{\abs{g_1}}+\sqrt{\abs{g_2}}\right\}
	\]
\end{lemma}
\begin{proof}
	From the definition of stationary solutions, we have that for all $x\in\mathbb{R}$ and $i=1,2$:
	\[
	f(x,U[g_i](x))=g_i.
	\]
	Suppose $g_1>g_2\geq0$. By \eqref{UC}, we also have that for all $x\in\mathbb{R}$, there exists some $\tilde{u}$, possibly depending on $x$, such that
	\begin{align*}
		\abs{g_1-g_2}&=f(x,U[g_1](x))-f(x,U[g_2](x)) \\
		&=f_u(x,U[g_2](x))[U[g_1](x)-U[g_2](x)]+f_{uu}(x,\tilde{u})\dfrac{1}{2}\abs{U[g_1](x)-U[g_2](x)}^2 \\
		&\geq\dfrac{\alpha}{2}\abs{U[g_1](x)-U[g_2](x)}^2,
	\end{align*}
	where the last inequality follows from \eqref{UC} and \eqref{S0}, which ensure that the linear term in the Taylor expansion is positive. Hence,
	\[
	\abs{U[g_1](x)-U[g_2](x)}\leq\sqrt{\dfrac{2}{\alpha}\abs{g_1-g_2}},
	\]
	and the same inequality can be shown for the case $g_1<g_2\leq0$ by a similar argument, mutatis mutandis, from which the desired inequality follows, since the RHS is independent of $x$. If $g_1,g_2$ have opposite signs, the second bound follows from the triangle inequality $\supnorm{u-v}\leq\supnorm{u}+\supnorm{v}$ obtained from comparing with the zero solution.
\end{proof}
Unlike the homogeneous case, neither the initial data nor the solutions are, in general, self-similar. However, by convexity, we can still solve the Cauchy problem corresponding to initial data of the form \eqref{gen riem data} by either a single shock wave or a rarefaction. Such solutions serve as building blocks for the general case.

\begin{theorem}[Solution of the generalised Riemann problem]\label{gen riem sol}
	Let $\overline{u}$ be of the form \eqref{gen riem data} (equivalently let $\overline{u}$ be such that $g(x,\overline{u}(x))$ is of the form \eqref{g gen riem data}), and let $\overline{u}(\overline{x}\pm)$ denote the right and left limits at the point of discontinuity. Define $g_l=g(\overline{x},\overline{u}(\overline{x}-))$ and $g_l=g(\overline{x},\overline{u}(\overline{x}+))$. We have two cases, depending on whether $g_l<g_r$ or vice versa.
	\begin{enumerate}
		\item If $g_l>g_r$, the entropy solution $u$ takes the form of a shock wave connecting the two stationary solutions travelling at Rankine-Hugoniot speed.
		
		\item If $g_l<g_r$: the solution $u$ takes the form of a rarefaction fan with characteristics of the system \eqref{char ode} emanating from $\overline{x}$.
	\end{enumerate}
	The function $g(x,u(x,t))$ is also Lipschitz continuous in time, with constant
	\[
	L=\max\left\{\theta\left(\pm\supnorm{\overline{u}}\right)\right\}\abs{g_l-g_r}.
	\]
\end{theorem}
\begin{proof}
	Suppose $g_l>g_r$. Let $f_l=\abs{g_l}, f_r=\abs{g_r}$. Indeed, $f(x,u_l(x))=f_l$ and $f(x,u_r(x))=f_r$. We explicitly define the curve of discontinuity and show that the resulting function is indeed an entropy solution of the Cauchy problem \eqref{claw} with generalised Riemann initial data \eqref{gen riem data}. Let $y(t)$ be the unique solution of
	\begin{align}
		\dot{y}(t)&=\dfrac{f_l-f_r}{u_l(y(t))-u_r(y(t))},\label{rh ode} \\
		y(0)&=\overline{x}.\label{rh init}
	\end{align} 
	Since $g_l>g_r$, we have that $\abs{u_l-u_r}$ is uniformly bounded away from zero by \eqref{FSP}. Hence, the right-hand side of the ODE \eqref{rh ode} is uniformly Lipschitz, thus a unique solution of the initial value problem \eqref{rh ode}-\eqref{rh init} exists by the standard theory of ordinary differential equations. Now, define
	\begin{equation}
		u(x,t)=
		\begin{cases}
			u_l(x)&\text{ if }x<y(t), \\
			u_r(x)&\text{ if }x>y(t).
		\end{cases}
	\end{equation}
	By construction, $u$ is a piecewise $C^1$ function that satisfies \eqref{claw} classically away from the curve $(y(t),t)$ and satisfies the Rankine-Hugoniot condition along it. Moreover, at points of discontinuity, we have that $u(x-,t)=u_l(x)>u_r(x)=u(x+,t)$, hence $u$ is an entropy solution. Since $u$ also obeys a maximum principle by \eqref{UC}, it is the \textit{unique} entropy solution \cite[Theorem 1]{kruzkov}.
	
	Next, suppose $g_l<g_r$. By the monotonicity of $g$ with respect to the second variable, $u_l(\overline{x})<u_r(\overline{x})$. Consider the family of initial value problems corresponding to \eqref{char ode} with $y(0)=\overline{x}$ and $z(0)\in[u_l(\overline{x}),u_r(\overline{x})]$. Denote the extremal characteristics corresponding to $z(0)=u_l(\overline{x}),u_r(\overline{x})$ by $y_l,y_r$ respectively. from \eqref{char ode} and \eqref{UC} we have that $\dot{y}_l(0)<\dot{y}_r(0)$.
	
	Since $f$ is conserved along characteristic trajectories by \eqref{f conserved}, it follows that $y_r(t)>y_l(t)$ for all $t>0$. We prove this by contradiction -- suppose, if possible, that $y_l(\tau)=y_r(\tau)=y$ for some $\tau>0$. Then, $\dot{y}_l(\tau)>\dot{y}_r(\tau)$. But by \eqref{char ode}, this implies $f_u(y,u_l(y))>f_u(y,u_r(y))$ which is impossible. Hence, $y_l,y_r$ can never meet.
	
	The analysis above also holds for all the intermediate trajectories. Hence, by continuous dependence, the family of ODEs described above `fill up' the entire domain lying between the curves $y_l(t),y_r(t)$. Define $u$ through these trajectories (specifically, the value of $z(t)$ for the corresponding curve $y$ such that $y(t)=x$) and by $u_l$ or $u_r$ appropriately outside the region. Then $u$ is Lipschitz for $t>0$, and satisfies \eqref{claw} classically pointwise almost everywhere where the derivatives exist. Once again, $u\in L^{\infty}$ by \eqref{UC}, and is thus the unique entropy solution of \eqref{claw} with the given initial data. 
	
	The Lipschitz time continuity follows from the maximum principle for $g$ and the finite speed of propagation property \eqref{FSP} of the flux.
\end{proof}
With these building blocks in hand, we construct approximate weak solutions $u$ of \eqref{claw} such that $g(x,u(x,t))$ is piecewise constant. Note that shock-type data already satisfy this latter condition, but the rarefaction fans need to be approximated by piecewise constant fans as is done in \cite{BRESSAN1992414}.
\begin{definition}\label{approx rarefaction}
	A $\delta$-fan solution of the Cauchy problem \eqref{claw} with initial data of the form \eqref{g gen riem data} is $u(x,t)$ defined as follows: if $g_r>g_l+\delta$, let $g_l=g_0<g_1<g_2\ldots<g_n<g_{n+1}=g_r$ be such that $g_r-g_{n}\leq\delta$, and $g_{i}-g_{i-1}=\delta$ otherwise for $1\leq i\leq n$. If $g_r\leq g_l+\delta$, then let $n=0$ and $g_1=g_r$. Then, let
	\[
	g(x,t)=
	\begin{cases}
		g_l&\text{ if }x<\gamma_0(t), \\
		g_{i}&\text{ if }\gamma_{i-1}(t)\leq x<\gamma_{i}(t)\text{ for }i=1,\ldots,n, \\
		g_r&\text{ if }\gamma_{n}(t)\leq x,
	\end{cases}
	\]
	where the curves $\gamma_{i}(t)$ are solutions of the respective ODEs
	\begin{align}
		\dot{\gamma}_i(t)&=\dfrac{\abs{g_{i+1}}-\abs{g_i}}{U[g_{i+1}](\gamma_{i}(t))-U[g_{i}](\gamma_{i}(t))} \label{rare RH}\\
		\gamma_{i}(0)&=\overline{x}. \label{init data}
	\end{align}
	While $\gamma_i$ thus defined are not necessarily entropic jumps (unless $g_r<g_l$), they nonetheless satisfy the Rankine-Hugoniot conditions. Thus, $\delta$-fans are weak solutions of \eqref{claw}.
\end{definition}
Finally, let us define what we mean by approximate solutions.
\begin{definition}\label{del approx}
	A $\delta$-approximate front tracking solution $u^{\delta}(x,t)$ defined on the domain $\Omega_T=\mathbb{R}\times[0,T]$ is a weak solution of \eqref{claw} such that $g(x,u^{\delta}(x,t))$ is piecewise constant in $\Omega_T$ in finitely many domains with Lipschitz boundaries, and $g(x,u^{\delta}(x+,t))-g(x,u^{\delta}(x-,t))\leq\delta$ for all $t>0$, where $u(x\pm,t)$ denote the left and right spatial limits at $(x,t)$ respectively.
\end{definition}
Thus, in particular, $\delta$-fan solutions of the generalised Riemann problem \eqref{gen riem data} are also $\delta$-approximate solutions.
 
\subsection{Front tracking}
We employ the generalised Riemann problem as described above to demonstrate the existence of entropy solutions by constructing a sequence of $\delta$-approximate front tracking solutions. Let $u_0\in L^{\infty}$ be given, such that $G_0(x)=g(x,u_0(x))$ is in $BV(\mathbb{R})$. Since $g$ is locally Lipschitz concerning $u$, this holds if, e.g. $u_0\in BV(\mathbb{R})$.

A quick overview of the approximation algorithm is as follows: since $G_0\in BV(\mathbb{R})$, we can approximate it by a piecewise constant functions $G^{\delta}_0$ with finitely many pieces such that the $L^1$ norm of $G_0-G^{\delta}_0$ is less than $\delta$. At each point of discontinuity, we encounter a Riemann problem that can be solved by Lemma~\ref{gen riem sol}. If the discontinuity is of rarefaction-type, we approximate the exact fan by a $\delta$-fan as per Definition~\ref{approx rarefaction}. Since each Riemann problem is solved by either a $\delta$-fan or an entropic shock, only finitely many fronts are created at the initial time, and we can solve \eqref{claw} approximately up to the first interaction time of fronts, which is positive by \eqref{FSP}. Every interaction strictly reduces the number of fronts, so only finitely many interactions need to be resolved, and thus a $\delta$-approximate front tracking solution exists globally in time.

More precisely, we have the following theorem.
\begin{theorem}[Existence of front tracking approximations]\label{ft existence}
	Let $u_0(x)$ be such that $G_0(x)=g(x,u_0(x))$ is of the form
	\[
	G_0(x)=
	\begin{cases}
		g_0&\text{ if }x<x_1, \\
		g_i&\text{ if }x\in[x_{i},x_{i+1})\text{ for }i=1,\ldots,n-1, \\
		g_{n+1}&\text{ if }x\geq x_n,
	\end{cases}
	\]
	where $x_0<x_1\ldots<x_n$ and $g_i\in\mathbb{R}$ for all $0\leq i\leq n$, i.e. $u_0(x)$ is such that $g(x,u_0(x))$ is piecewise constant. Then, for all $\delta>0$ and $T>0$, a $\delta$-approximate front tracking solution $u^{\delta}$ in the sense of Definition~\ref{del approx} exists on $\Omega_T$ with $u^{\delta}(x,0)=u_0(x)$.
	
	Furthermore, the spatial total variation of $g(\cdot,u^{\delta}(\cdot,t))$ is non-increasing in $t$. Thus, in particular, the total variation at any time $t>0$ is bounded by the initial total variation of $G_0(x)$, which is precisely $\abs{g_1-g_0}+\ldots+\abs{g_{n+1}-g_n}$.
	
	The front tracking solution is also such that for $h>0$,
	\begin{equation}\label{Lip L1}
		\pnorm{L^1(\mathbb{R})}{g(\cdot,u^{\delta}(\cdot,t+h))-g(\cdot,u^{\delta}(\cdot,t))}\leq L\pnorm{TV}{G_0}h.
	\end{equation}
\end{theorem}
\begin{proof}
	At each point of discontinuity $x_i$, we can solve the associated generalised Riemann problem. Let us order the fronts from left to right as $\gamma_k(t)$ with $k$ ranging from $0$ to some finite $N$, let $u^\delta(x,t)$ denote the approximate solution, and let $\overline{g}^{\delta}(x,t)=g(x,u^{\delta}(x,t))$. Thus, for some constants $\overline{g}_k$, we have that for $t>0$,
	\[
	\overline{g}^{\delta}(x,t)=
	\begin{cases}
		\overline{g}_0&\text{ if }x<\gamma_0(t), \\
		\overline{g}_k&\text{ if }x\in[\gamma_{k-1}(t),\gamma_k(t))\text{ for }1\leq k\leq N-1, \\
		\overline{g}_{N+1}&\text{ if }x\geq\gamma_N(t).
	\end{cases}
	\]
	Now, $\gamma_{k}(0)\leq\gamma_{j}(0)$ for $j>k$, and by \eqref{FSP}, $\gamma_k(t)<\gamma_j(t)$ for $t$ small enough, say $t\leq\tau$, where $\tau$ is the first positive time that at least two distinct fronts meet. Thus, $\overline{g}^{\delta}$ as above is well-defined for $t\leq\tau$.
	
	Suppose the interacting fronts are indexed from $\gamma_{k}$ to $\gamma_{k+m}$, where $m\geq1$, i.e.
	\[
	\gamma_{k-1}(\tau)<\gamma_k(\tau)=\gamma_{k+1}(\tau)=\ldots=\gamma_{k+m}(\tau)<\gamma_{k+m+1}(\tau).
	\]
	Let us denote the point of interaction $\gamma_k(\tau)=\rho$. Since, by construction, we have that for $t<\tau$,
	\[
	\gamma_k(t)<\gamma_{k+1}(t)<\ldots<\gamma_{k+m}(t),
	\]
	and the curves are smooth by \eqref{rh ode} and/or \eqref{rare RH}, we must have that
	\[
	\dot{\gamma}_{k}(\tau)\geq\dot{\gamma}_{k+1}(\tau)\geq\ldots\geq\dot{\gamma}_{k+m}(\tau),
	\]
	and hence by \eqref{rh ode} and/or \eqref{rare RH} again,
	\[
	\dfrac{\abs{\overline{g}_{k+1}}-\abs{\overline{g}_k}}{U[\overline{g}_{k+1}](\rho)-U[\overline{g}_{k}](\rho)}\geq\dfrac{\abs{\overline{g}_{k+2}}-\abs{\overline{g}_{k+1}}}{U[\overline{g}_{k+2}](\rho)-U[\overline{g}_{k+1}](\rho)}\geq\ldots\geq\dfrac{\abs{\overline{g}_{k+m+1}}-\abs{\overline{g}_{k+m}}}{U[\overline{g}_{k+m+1}](\rho)-U[\overline{g}_{k+m}](\rho)}.
	\]
	Suppose $m=1$. We claim that $\overline{g}_{k+2}<\overline{g}_{k}$. If this weren't the case, then either $\overline{g}_{k+1}\in(\overline{g}_{k},\overline{g}_{k+2})$ or $\overline{g}_{k+1}\in[\overline{g}_k,\overline{g}_{k+2}]^c$. In the first case, by the uniform convexity of $f(\rho,\cdot)$ we have that
	\[
	\dfrac{\abs{\overline{g}_{k+1}}-\abs{\overline{g}_k}}{U[\overline{g}_{k+1}](\rho)-U[\overline{g}_{k}](\rho)}<\dfrac{\abs{\overline{g}_{k+2}}-\abs{\overline{g}_{k+1}}}{U[\overline{g}_{k+2}](\rho)-U[\overline{g}_{k+1}](\rho)},
	\]
	which is impossible. Now, suppose $\overline{g}_{k+1}<\overline{g}_{k}$; the other case can be handled similarly. Then, by \eqref{UC} again,
	\[
	\abs{\overline{g}_{k+1}}+\dfrac{\abs{\overline{g}_{k+2}}-\abs{\overline{g}_{k+1}}}{U[\overline{g}_{k+2}](\rho)-U[\overline{g}_{k+1}](\rho)}\left(U[\overline{g}_k](\rho)-U[\overline{g}_{k+1}](\rho)\right)>\abs{\overline{g}_{k}},
	\]
	which implies that
	\[
	\abs{\overline{g}_{k+1}}-\abs{\overline{g}_k}>\dfrac{\abs{\overline{g}_{k+2}}-\abs{\overline{g}_{k+1}}}{U[\overline{g}_{k+2}](\rho)-U[\overline{g}_{k+1}](\rho)}\left(U[\overline{g}_{k+1}](\rho)-U[\overline{g}_k](\rho)\right).
	\]
	Now, since $\overline{g}_{k+1}<\overline{g}_k$, and $g$ is monotone, we have that $U[\overline{g}_{k+1}](\rho)-U[\overline{g}_k](\rho)<0$. Hence, it follows once again that
	\[
	\dfrac{\abs{\overline{g}_{k+1}}-\abs{\overline{g}_k}}{U[\overline{g}_{k+1}](\rho)-U[\overline{g}_{k}](\rho)}<\dfrac{\abs{\overline{g}_{k+2}}-\abs{\overline{g}_{k+1}}}{U[\overline{g}_{k+2}](\rho)-U[\overline{g}_{k+1}](\rho)}.
	\]
	Hence, $\overline{g}_k>\overline{g}_{k+2}$. Now, if $m>1$, we can do this for every triple $\overline{g}_{j},\overline{g}_{j+1},\overline{g}_{j+2}$ with $j<k+m-1$. Thus, if $m$ is odd, it follows that $\overline{g}_k>\overline{g}_{k+m+1}$. On the other hand, if $m$ is even, then $\overline{g}_{k+1}>\overline{g}_{k+m+1}$, but by construction $\overline{g}_{k+1}\leq\overline{g}_{k}+\delta$. 
	
	Hence, in either case, $\overline{g}_{k+m+1}\leq\overline{g}_{k}+\delta$, the constructed function $u^{\delta}$ is still a $\delta$-approximate solution up to $t=\tau$, and the generalised Riemann problem at $(\rho,\tau)$ is solved by a single front. We can repeat the procedure now, solving up to the next interaction time, and continuing each time by a single front. The number of fronts is thus a non-increasing function of time, reducing by at least one at each time of interaction. Since we start with $N$ fronts, only finitely many interactions may take place (at most $N-1$, to be precise), and a $\delta$-approximate solution can be defined on arbitrary domains of the form $\Omega_T$.
	
	Since the total variation of $\overline{g}^{\delta}$ only changes at interaction points, and the continuation involves only a single front each time, it follows that the total variation of $\overline{g}^{\delta}$ does not increase with time. Lipschitz continuity concerning the $L^1$ norm for $g(x,u^{\delta}(x,t))$ follows from this total variation diminishing (TVD) property and \eqref{Lip L1} of Theorem~\ref{gen riem sol}. This completes the proof.
\end{proof}
With such $\delta$-approximate solutions defined for piecewise stationary initial data, we now turn to the general Cauchy problem.

\section{The Cauchy problem}\label{cauchy prob}
In this section, we prove well-posedness of the Cauchy problem \eqref{claw} by a front tracking argument. In order to do this, we define a special class of functions. For arbitrary $\delta>0$, let $\mathcal{A}_{\delta}$ be the subset of $L^{\infty}(\mathbb{R})$ consisting of functions $u_0$ such that $G_0(x)=g(x,u_0(x))$ is piecewise constant with finitely many discontinuities and takes values in the set $\delta\mathbb{Z}=\{\delta z:z\in\mathbb{Z}\}$. That is, $g(x,u_0(x))$ takes on values in the discrete additive subgroup of $\mathbb{R}$ generated by $\delta>0$.

In contrast with the original discretisation of Dafermos \cite{polygon}, in which the domain of $f$ is discretised, we discretise the range instead. This is analogous to the change of perspective from Riemann to Lebesgue integration.

\subsection{A priori estimates}
The following lemma trivially follows from Theorem~\ref{ft existence} as a special case and is presented without proof.
\begin{lemma}\label{discrete approx}
	For $u_0\in\mathcal{A}_{\delta}$, the $\delta$-approximate (front tracking) solution $u^{\delta}$ as constructed in Theorem~\ref{ft existence} is such that $u^{\delta}(\cdot,t)\in\mathcal{A}_{\delta}$ for all $t\geq0$.
\end{lemma}
Next, we define an approximation of the convex flux $f$, concerning which $\delta$-approximate solutions are entropy solutions. This helps us pass to the limit in the approximation parameter $\delta$ and obtain an entropy solution of the original Cauchy problem \eqref{claw}.
\begin{definition}\label{approx flux}
	The $\delta$-approximate flux $f^{\delta}$ corresponding to $f$ in \eqref{claw} is defined as follows:
	\[
	f^{\delta}(x,u)=
	\begin{cases}
		&f(x,u)\text{ if }f(x,u)\in\delta\mathbb{Z}, \\[1.5ex]
		&\delta\abs{z}+\dfrac{\delta\operatorname{sgn}(z)(u-U[\delta z](x))}{U[\delta(z+1)](x)-U[\delta z](x)}\text{ if }g(x,u)\in(\delta z,\delta(z+1))\text{ for some }z\in\mathbb{Z}.
	\end{cases}
	\]
	Note that, for $g(x,u)\in(\delta z,\delta(z+1))$, when $z\geq0$, we can equivalently write
	\[
	f^{\delta}(x,u)=\delta(z+1)+\dfrac{\delta(u-U[\delta(z+1)](x))}{U[\delta(z+1)](x)-U[\delta z](x)},
	\]
	and a similar expression, mutatis mutandis, can be obtained when $z<0$. That is, for each $x\in\mathbb{R}, f^{\delta}(x,\cdot)$ is a piecewise linear interpolation of $f(x,\cdot)$ matching exactly wherever $f$ takes values in $\delta\mathbb{Z}$, or rather $\delta\mathbb{N}$, since $f\geq0$. The function $f^{\delta}$ as defined is $C^2$ (respectively, locally Lipschitz) concerning the first (respectively, second) argument.
\end{definition}
Since the $\delta$-approximate flux is linear in the conserved variable between breakpoints, it follows that the $\delta$-approximate front tracking solution $u^{\delta}$ corresponding to any $u_0\in\mathcal{A}_{\delta}$ is an entropy solution of the conservation law
\begin{equation}\label{approx claw}
	\begin{split}
		u^{\delta}_t+f^{\delta}(x,u^{\delta})_x&=0, \\
		u^{\delta}(x,0)&=u_0(x).
	\end{split}
\end{equation}
Therefore, for any $k\in\mathbb{R}$, we have that for all $\varphi\in C_c^{\infty}(\mathbb{R}\times[0,\infty))$ such that $\varphi\geq0$,
\begin{equation}\label{approx entropy}
	\begin{split}
		&\iint\abs{u^{\delta}(x,t)-k}\varphi_t+\operatorname{sgn}(u^{\delta}(x,t)-k)\left[f^{\delta}(x,u^{\delta}(x,t))-f^{\delta}(x,k)\right]\varphi_xdxdt \\
		\geq&\iint\operatorname{sgn}(u^{\delta}(x,t)-k)f_x^{\delta}(x,k)\varphi(x,t) dxdt-\int\abs{u_0-k}\varphi(x,0)dx.
	\end{split}
\end{equation}
To pass to the limit, we need estimates on both $f^{\delta}$ and its spatial derivative $f^{\delta}_x$, and in particular their respective relations with $f,f_x$ as $\delta\to0$.
\begin{lemma}[Convergence of the approximate flux]\label{flux convergence}
	As $\delta\to0$, the fluxes $f^{\delta}\to f$ uniformly on subsets of the form $\mathbb{R}\times[-M,M]$. Furthermore, $f_x^{\delta}\to f_x$ uniformly on compact subsets of $\mathbb{R}^2$.
\end{lemma}
\begin{proof}
	Let $[-M,M]\subset\mathbb{R}$. By \eqref{UC}-\eqref{FSP} and Lemma~\ref{cont dependence}, we have that for $\abs{u}\leq M$,
	\begin{equation}\label{flux der 1}
		\begin{split}
			\abs{f^{\delta}(x,u)-f(x,u)}&\leq\abs{f(x,U[\delta z](x))-f(x,u)}+\abs{\dfrac{\delta\left(u-U[\delta z](x)\right)}{U[\delta(z+1)](x)-U[\delta z](x)}} \\
			&\leq\sqrt{\dfrac{2\delta}{\alpha}}\pnorm{L^{\infty}(\mathbb{R}\times[-M-\delta,M+\delta])}{f_u}+\delta \\
			&\leq\sqrt{\dfrac{2\delta}{\alpha}}\left(1+\max\{\theta(M+\delta),\theta(-M-\delta)\}\right)+\delta,
		\end{split}
	\end{equation}
	which proves the required uniform convergence result. Now, suppose $g(x,u)=\delta z$ for some integer $z$; without loss of generality, assume $z>0$. Then, for $\abs{h}$ small enough, $g(x+h,u)\in(\delta(z-1),\delta(z+1))$. Hence, by Definition~\ref{approx flux}, we have that
	\[
	f^{\delta}(x+h,u)=
	\begin{cases}
		&\delta{z}+\dfrac{\delta(u-U[\delta z](x+h))}{U[\delta(z+1)](x+h)-U[\delta z](x+h)}\text{ if }g(x+h,u)\geq\delta z, \\[2ex]
		&\delta{z}+\dfrac{\delta(u-U[\delta z](x+h))}{U[\delta z](x+h)-U[\delta(z-1)](x+h)}\text{ if }g(x+h,u)\leq\delta z.
	\end{cases}
	\]
	Note that the definitions match if $g(x+h,u)=\delta z$. We can write the above in compact form as follows:
	\begin{align*}
		f^{\delta}(x+h,u)=\delta{z}+\dfrac{\operatorname{sgn}(g(x+h,u)-\delta z)\delta(u-U[\delta z](x+h))}{(U[\delta(z+\operatorname{sgn}(g(x+h,u)-\delta z)](x+h)-U[\delta z](x+h))}.
	\end{align*}
	Now, $f^{\delta}(x,u)=f(x,u)=\delta z$ and $u=U[\delta z](x)$, by assumption. Hence, for $\abs{h}$ small enough,
	\begin{align*}
		\dfrac{f^{\delta}(x+h,u)-f^{\delta}(x,u)}{h}&=\dfrac{f^{\delta}(x+h,u)-f(x,u)}{h} \\
		&=\dfrac{f^{\delta}(x+h,u)-\delta z}{h} \\
		&=-\dfrac{U[\delta z](x+h)-U[\delta z](x)}{h}\dfrac{\delta}{\abs{U[\delta(z\pm1)](x+h)-U[\delta z](x+h)}},
	\end{align*}
	and hence, as $h\to0$, we have that by the mean value theorem, for some $\lambda\in(-1,1)$:
	\begin{align*}
		f_x^{\delta}(x,u)=&-\partial_xU[\delta z](x)f_u(x,U[\delta(z+\lambda)](x)) \\
		=&-\partial_xU[\delta z](x)f_u(x,U[\delta z](x)) \\
		&+\partial_xU[\delta z](x)\left(f_u(x,U[\delta z](x))-f_u(x,U[\delta(z+\lambda)](x))\right) \\
		=&+f_x(x,u) \\
		&+\partial_xU[\delta z](x)\left(f_u(x,U[\delta z](x))-f_u(x,U[\delta(z+\lambda)](x))\right).
	\end{align*}
	where $\lambda\in(-1,1)$. Hence, on compact subsets $K\subset\mathbb{R}^2$ with $K\subseteq K_x\times K_u$ for compact intervals $K_x,K_u\subset\mathbb{R}$, we have that by Lemma~\ref{cont dependence}
	\begin{align*}
		\abs{f_x^{\delta}(x,u)-f_x(x,u)}&\leq\sqrt{\dfrac{2\delta}{\alpha}}\pnorm{L^{\infty}(K_x)}{\partial_xU[\delta z](\cdot)}\pnorm{L^{\infty}(K_u)}{f_{uu}(x,\cdot)}
	\end{align*}
	If, on the other hand, $g(x,u)\in(\delta z,\delta(z+1))$ for some integer $z$ that we take to be non-negative without loss of generality, then for small enough $\abs{h}$, we have that $f(x+h,u)\in(\delta z,\delta(z+1))$ as well. Hence,
	\[
	f^{\delta}(x,u)=\delta{z}+\dfrac{\delta(u-U[\delta z](x))}{U[\delta(z+1)](x)-U[\delta z](x)},
	\]
	and
	\[
	f^{\delta}(x+h,u)=\delta{z}+\dfrac{\delta(u-U[\delta z](x+h))}{U[\delta(z+1)](x+h)-U[\delta z](x+h)}.
	\]
	Now, we can write
	\[
	f^{\delta}(x+h,u)=\delta{z}+\dfrac{\delta(u-U[\delta z](x))+\delta(U[\delta z](x)-U[\delta z](x+h))}{U[\delta(z+1)](x+h)-U[\delta z](x+h)}.
	\]
	Hence for small enough $\abs{h}$,
	\begin{align*}
		&\dfrac{f^{\delta}(x+h,u)-f^{\delta}(x,u)}{h} \\
		=&-\dfrac{U[\delta z](x+h)-U[\delta z](x)}{h}\dfrac{\delta}{U[\delta(z+1)](x+h)-U[\delta z](x+h)} \\
		&+\dfrac{\delta(u-U[\delta](z))\left\{U[\delta(z+1)](x)-U[\delta z](x)-U[\delta(z+1)](x+h)+U[\delta z](x+h)\right\}}{h(U[\delta(z+1)](x+h)-U[\delta z](x+h))(U[\delta(z+1)](x)-U[\delta z](x))},
	\end{align*}
	and so, by the mean value theorem again, we have that for some $\lambda\in(0,1)$:
	\begin{align*}
		f_x^{\delta}(x,u)=&\lim_{h\to0}\dfrac{f^{\delta}(x+h,u)-f^{\delta}(x,u)}{h} \\
		=&-\partial_xU[\delta z](x)f_u(x,U[\delta(z+\lambda)](x)) \\
		&+\dfrac{\delta(u-U[\delta z](x))}{(U[\delta(z+1)](x)-U[\delta z](x))^2}\left\{\partial_xU[\delta z](x)-\partial_xU[\delta(z+1)](x)\right\}.
	\end{align*}
	Now, $u\in(U[\delta z](x),U[\delta(z+1)](x))$, hence
	\[
	0<\dfrac{u-U[\delta z](x)}{U[\delta(z+1)](x)-U[\delta z](x))}<1,
	\]
	and therefore, for some $\lambda_1\in(0,1)$,
	\begin{align*}
		\abs{f_x^{\delta}(x,u)-f_x(x,u)}\leq&\abs{\partial_xU[\delta z](x)}\abs{f_u(x,U[\delta(z+1)](x))-f_u(x,U[\delta z](x))} \\
		&+\abs{f_u(x,U[\delta(z+\lambda_1)])\left\{\partial_xU[\delta z](x)-\partial_xU[\delta(z+1)](x)\right\}}.
	\end{align*}
	However,
	\begin{align*}
		&f_u(x,U[\delta(z+\lambda_1)])\partial_xU[\delta z](x) \\
		=&+f_u(x,U[\delta z](x))\partial_xU[\delta z](x) \\
		&-\left[f_u(x,U[\delta z](x))-f_u(x,U[\delta(z+\lambda_1)](x))\right]\partial_xU[\delta z](x) \\
		=&-f_x(x,U[\delta z](x)) \\
		&-\left[f_u(x,U[\delta z](x))-f_u(x,U[\delta(z+\lambda_1)](x))\right]\partial_xU[\delta z](x),
	\end{align*}
	and similarly
	\begin{align*}
		&f_u(x,U[\delta(z+\lambda_1)])\partial_xU[\delta(z+1)](x) \\
		=&-f_x(x,U[\delta(z+1)]) \\
		&-\left[f_u(x,U[\delta(z+1)](x))-f_u(x,U[\delta(z+\lambda_1)](x))\right]\partial_xU[\delta(z+1)](x),
	\end{align*}
	hence, over the compact set $K\subseteq K_x\times K_u$, by Lemma~\ref{cont dependence},
	\begin{align*}
		&\abs{f_u(x,U[\delta(z+\lambda_1)])\left\{\partial_xU[\delta z](x)-\partial_xU[\delta(z+1)](x)\right\}} \\
		&\leq\sqrt{\dfrac{2\delta}{\alpha}}\left\{\pnorm{L^{\infty}(K)}{f_{xu}}+2\pnorm{L^{\infty}(K)}{f_{uu}}\left(\sup_{y\in K_u}\pnorm{L^{\infty}(K_x)}{\partial_xU[y]}\right)\right\},
	\end{align*}
	and therefore
	\begin{align}\label{flux der 2}
		\abs{f_x^{\delta}(x,u)-f_x(x,u)}\leq\sqrt{\dfrac{2\delta}{\alpha}}\left\{\pnorm{L^{\infty}(K)}{f_{xu}}+3\pnorm{L^{\infty}(K)}{f_{uu}}\left(\sup_{y\in K_u}\pnorm{L^{\infty}(K_x)}{\partial_xU[y]}\right)\right\}
	\end{align}
	where all the suprema are finite by \eqref{C2} and compactness of the given set $K$. From \eqref{flux der 1} and \eqref{flux der 2}, we conclude that $f_x^{\delta}\to f_x$ uniformly on compact sets, as claimed.
\end{proof}

\subsection{Passing to the limit}
Given any $u_0\in L^1(\mathbb{R})\cap L^{\infty}(\mathbb{R})$ such that $G_0(x)=g(x,u_0(x))$ is of bounded variation, we can approximate $u_0$ with respect to the $L^1$ norm by functions in $\mathcal{A}_{\delta}$ to an arbitrary degree by choosing $\delta>0$ small enough. Since the front tracking solutions $u{\delta}$ are such that $g^{\delta}(x,t)=g(x,u^{\delta}(x,t))$ is total variation diminishing and uniformly Lipschitz continuous in time concerning the $L^1$ norm, we can pass to the limit in a subsequence that converges in $L^1$ on compact subsets of $\mathbb{R}\times[0,\infty)$. Without loss of generality, we may assume that the convergence holds pointwise almost everywhere. More precisely, we have the following result.
\begin{theorem}[Existence and uniqueness]
	Let $u_0\in L^{\infty}(\mathbb{R})$ such that $G_0(x)=g(x,u_0(x))$ is of bounded variation. Then there exists a unique entropy solution $u\in L^{\infty}(\mathbb{R}^2_{+})$ with $u\in C([0,\infty),L^1_{loc}(\mathbb{R}))$ to the Cauchy problem \eqref{claw} with initial value $u_0$ in the sense of \eqref{entropy}.
\end{theorem}
\begin{proof}	
	Suppose $u_0\in L^{1}\cap L^{\infty}(\mathbb{R})$. By \eqref{S0} and \eqref{FSP}, we can approximate $G_0$ by compactly supported and piecewise constant functions $G^{\delta}_0$ such that $U[G^{\delta}_0(\cdot)](\cdot)\in\mathcal{A}_{\delta}$ and is also compactly supported, with finitely many discontinuities. For initial values $u_0^{\delta}(x)=U[G^{\delta}_0](x)$, a $\delta$-approximate front tracking solution exists by Theorem~\ref{ft existence}. In particular, they also satisfy the approximate entropy inequality \eqref{approx entropy}.
	
	Let $\delta_j=1/j$, so that $\delta_j\to0$ and $G_0^{\delta_j}\to G_0$ in $L^1$. Now, the functions $g^{\delta_J}(x,t)=g(x,u^{\delta}(x,t))$ satisfy a uniform spatial total variation bound. Hence, for each fixed time, we can extract a subsequence, still denoted by $\delta_j$, such that $g^{\delta_j}(\cdot,t)$ converges in $L^1$ on compact intervals by Helly's theorem. They are also uniformly Lipschitz continuous in time with respect to the $L^1$ norm. Thus, by a standard diagonalisation argument as laid out for instance in \cite{timecompactness}, we can extract a subsequence, still denoted by $\delta_j$, such that $g^{\delta_j}\to g$ in $L^1_{loc}(\mathbb{R}\times[0,\infty))$, i.e. in $L^1$ on each compact set of $\Omega_{\infty}=\mathbb{R}\times[0,\infty)$.
	
	A brief sketch of the diagonal argument: at each rational time $t_i$ for some enumeration of the non-negative rationals, we can extract successively convergent subsequences (by the total variation bound and Helly's theorem). The diagonal subsequence of this series of subsequences converges in $L^{1}_{loc}(\mathbb{R})$ for each rational time $t_i$, thus by density and the uniform Lipschitz time-continuity, for all times. Note that we could have started with any sequence $\delta_j\to0$. The limit $g$ also inherits the Lipschitz time-continuity concerning the $L^1$ norm that $g^{\delta}$ possesses.
	
	Passing to a subsequence if necessary, assume that $g^{\delta}\to g$ pointwise almost everywhere. Hence, the $\delta$-approximate solutions, defined as
	\[
	U[g^{\delta}(x,t)](x)=u^{\delta}(x,t),
	\]
	also converge pointwise almost everywhere. Furthermore, by \eqref{S0}, \eqref{UC}, and the generalised maximum principle for the front tracking solutions in Theorem~\ref{ft existence}, $u^{\delta}$ satisfies uniform $L^{\infty}$ bounds as well. Thus, by the dominated convergence theorem, $u^{\delta}\to u$ in $L^{1}_{loc}$. Now, the approximate solutions $u^{\delta}$ satisfy \eqref{claw} as well as \eqref{approx entropy}. Since the fluxes $f^{\delta}$ also converge uniformly to $f$ as $\delta\to0$ for $(x,u)\in\mathbb{R}\times[-\supnorm{u},\supnorm{u}]$, with $f_x^{\delta}$ converging uniformly on compact sets, we can pass to the limit by the dominated convergence theorem. Thus, for any $\varphi\in C_c^{\infty}(\Omega_{\infty})$ with $\varphi\geq0$, the entropy inequality \eqref{entropy} is satisfied, i.e. for all $k\in\mathbb{R}$,
	\begin{align*}
		&\iint\abs{u(x,t)-k}\varphi_t+\operatorname{sgn}(u(x,t)-k)\left[f(x,u(x,t))-f(x,k)\right]\varphi_xdxdt \\
		\geq&\iint\operatorname{sgn}(u(x,t)-k)f_x(x,k)\varphi(x,t) dxdt-\int\abs{u_0-k}\varphi(x,0)dx.
	\end{align*}
	Since $u\in L^{\infty}$ and $f$ satisfies \eqref{FSP}, furthermore, we have uniqueness by \cite[Theorem 1]{kruzkov}. More precisely, we have $L^1$ stability with respect to a finite domain of dependence; bounded entropy solutions of \eqref{claw} $u,v$ with respective initial values $u_0,v_0$ satisfy
	\begin{align}\label{dom of dep}
		\int_{-R}^{R}\abs{u(x,t)-v(x,t)}dx\leq\int_{-R-Lt}^{R+Lt}\abs{u_0(x)-v_0(x)}dx
	\end{align}
	for all $t,R>0$, where $L$ is the common upper bound of $\theta(\pm\supnorm{u}),\theta(\pm\supnorm{v})$, which are finite by \eqref{FSP}. In particular, this tells us that $u(\cdot,t)\in L^1(\mathbb{R})$ with uniform $L^1$ bound for all times $t\geq0$. Hence, the sequence of front tracking approximations has a unique limit. For the general case, we can approximate $u_0$ in turn by multiplying with cut-off functions; by \eqref{FSP} and \eqref{dom of dep}, then, the sequence of solutions converges in $L^1_{loc}$ on the domain $\Omega_{\infty}$.
\end{proof}
Finally, note that for fluxes which are merely strictly convex, we can add a term $\varepsilon u^2$. Then, we can solve the Cauchy problem by front tracking for $f_{\varepsilon}(x,u)=f(x,u)+\varepsilon u^2$ which satisfies all our assumptions instead of our original flux function. If we let $\varepsilon\to0$, then, we will obtain a solution to the original Cauchy problem. Thus, \eqref{UC} can be relaxed to strict convexity.

\section{Application of the method}\label{app}
This front tracking method for conservation laws with a spatially heterogeneous flux can be used to prove the existence of entropy solutions in the case of fluxes for which all classical methods fail, primarily due to their reliance on the maximum principle, e.g. the compensated compactness technique of \cite{ConLawHJB}. The purely $L^1$ approach of \cite{dalibard} via the kinetic formulation, on the other hand, requires that $f_x$, or more generally the spatial divergence in several dimensions, be a uniformly bounded function of its arguments.

Now, if $f(x,u)=xu^2$, then it is trivial to see that $f$ does not satisfy any of the conditions necessary to conclude the existence of entropy solutions from the existing theory; note that $f$ is also not `weakly genuinely non-linear' in the sense of \cite{ConLawHJB}, since $f(0,\cdot)$ is uniformly zero. As we will show below, entropy solutions of \eqref{claw} with $f(x,u)=xu^2$ do not even satisfy the maximum principle, hence finite volume schemes such as \cite{sylla2024convergencefinitevolumescheme} also cannot be used to conclude the existence of entropy solution.

The closest paper in the literature is \cite{otto} which precisely deals with a multiplicative flux. However, the existence of entropy solutions is merely \emph{assumed} and not \emph{proved} here, since the paper primarily deals with regularising effects and not existence of solutions.

\subsection{Unbounded solutions}
Although this flux also violates our assumption \eqref{UC} at $x=0$, we can adapt the front tracking method to deal with this issue. First, let us demonstrate that a priori $L^{\infty}$ estimates genuinely fail for this equation. Consider the unbounded function
\[
u(x,t)=
\begin{cases}
	-1/\sqrt{x}&\text{ if }x>\gamma(t), \\
	0&\text{ otherwise},
\end{cases}
\]
where $\gamma(t)$ is the Lipschitz curve
\[
\gamma(t)=
\begin{cases}
	(t-2)^2/4&\text{ if }t\leq2, \\
	0&\text{ if }t>2.
\end{cases}
\]
At least for $t<2$, it is trivial to see that $u(x,t)$ is an entropy solution of \eqref{claw} with bounded initial data. It solves the PDE pointwise on either side of the curve of discontinuity $\gamma(t)$, which in turn satisfies the Rankine-Hugoniot condition with the right sign. Indeed, the initial data $u_0(x)=u(x,0)$ is essentially a generalised Riemann data of the form \eqref{gen riem data}, and thus $u(x,t)$ for $t<2$ is also a front tracking solution.

We claim that $u(x,t)$ is also an entropy solution for all $t\geq2$. It is enough to show that the stationary solution $u(x,t+2)$ is an entropy solution of \eqref{claw}, and furthermore it is enough to prove the entropy inequality for compactly supported non-negative test functions in the \emph{open} upper half plane, since the function is trivially $L^1_{loc}$ continuous.

Thus, let $v(x,t)=u(x,t)$, and note that $v$ is a classical solution of \eqref{claw} for $x\neq0$. For any compactly supported $\varphi\geq0$ and $k\in\mathbb{R}$, note that for all $\varepsilon>0$,
\begin{align*}
	&\iint_{\mathbb{R}^2_+}\abs{v-k}\varphi_t+\operatorname{sgn}(v-k)\left\{\left[f(x,u)-f(x,k)\right]\varphi_x-f_x(x,k)\varphi\right\}dxdt \\
	=&\int_0^{\infty}\int_{0}^{\varepsilon}\abs{v-k}\varphi_t+\operatorname{sgn}(v-k)\left\{\left[v^2-k^2\right]x\varphi_x-k^2\varphi\right\}dxdt \\
	&+\int_0^{\infty}\int_{\varepsilon}^{\infty}\abs{v-k}\varphi_t+\operatorname{sgn}(v-k)\left\{\left[v^2-k^2\right]x\varphi_x-k^2\varphi\right\}dxdt
\end{align*}
Since all the terms are integrable, the first term vanishes as $\varepsilon\to0$. Recall however that $v$ is a classical solution of \eqref{claw} for $x>0$, and thus, applying integration by parts on the second integral yields
\begin{align*}
	&\int_0^{\infty}\int_{\varepsilon}^{\infty}\abs{v-k}\varphi_t+\operatorname{sgn}(v-k)\left\{\left[v^2-k^2\right]x\varphi_x-k^2\varphi\right\}dxdt \\
	=&\int_{0}^{\infty}\operatorname{sgn}\left(k+1/\sqrt{\varepsilon}\right)\varepsilon\left(\dfrac{1}{\varepsilon}-k^2\right)\varphi(\varepsilon,t)dt.
\end{align*}
Now, given any fixed $k\in\mathbb{R}$, this integral is clearly non-negative for all $\varepsilon$ small enough. Hence, we obtain the entropy inequality by letting $\varepsilon\to0$. Since $\varphi\in C_c^{\infty}(\mathbb{R}^2_+)$ and $k\in\mathbb{R}$ were arbitrary, this proves that $v$, and thus $u$, is indeed an entropy solution.

We remark that a similar construction can be carried out, mutatis mutandis, for solutions supported in $x<0$. Indeed, this method of proof works for any generalised Riemann problem whose $L^{\infty}$ norm blows up. This demonstration also naturally leads into our front tracking approximation method for solving the Cauchy problem for the flux $f(x,u)=xu^2$. If $g(x,u_0(x))\in BV(\mathbb{R})$, then we can construct $\delta$-approximate front tracking solutions by setting $g_0=0$ in a small enough neighbourhood of $x=0$, discretising as usual otherwise, and solving for front interactions, with the only change being: fronts that hit $x=0$ simply remain there. With this one minor tweak, we can conclude the existence of entropy solutions for the Cauchy problem even for fluxes such as $f(x,u)=xu^2$.

\section{Acknowledgements}
The author would like to thank the Department of Atomic Energy, Government of India, for their support under project no. 12-R\&D-TFR-5.01-0520, and Dr. Shyam Sundar Ghoshal for his helpful guidance.

\bibliographystyle{plain}
\bibliography{citations}
\end{document}